\documentclass[12pt]{amsart}

\usepackage{amsmath,amsthm,amssymb,bm}
\usepackage{hyperref}
\usepackage{cleveref}
\usepackage{graphicx,color}
\usepackage{a4wide}

\usepackage{tikz}
\usetikzlibrary{calc}
\usepackage{ytableau}
\numberwithin{equation}{section}

\usepackage{enumitem}
\setlist[itemize]{label=\textbullet}

\newtheorem*{thm}{Theorem}
\newtheorem*{defn}{Definition}
\newtheorem{lem}{Lemma}

\newtheorem*{ques}{Question}

\newcommand\sym{\mathsf{Sym}}

\newcommand\cell{\mathsf{cell}}

\newcommand\Sym{\sym}
\newcommand\calC{\mathcal{C}}
\newcommand\SSYT{\mathsf{SSYT}}
\newcommand\RSPP{\mathsf{RSPP}}

\newcommand\wt{\mathrm{wt}}
\newcommand\Fix{\mathsf{Fix}}

\newcommand{\brick}[3]{ \draw (#1,#2) rectangle ++(#3,1);}

\title[Sign-reversing involution for the Schur antipode]{A sign-reversing involution for the antipode of Schur functions}

\author{Younggwang Cho}
\address{Department of Mathematics, Sungkyunkwan University, Suwon, South Korea}
\email{brglory@g.skku.edu}

\author{Byung-Hak Hwang}
\address{Center for Artificial Intelligence and Natural Sciences, Korea Institute for Advanced Study, Seoul, South Korea}
\email{byunghakhwang@gmail.com}

\author{Hojoon Lee}
\address{Department of Mathematics, Sungkyunkwan University, Suwon, South Korea}
\email{hojoon1101@g.skku.edu}

\begin{document}
\ytableausetup
{mathmode, boxframe=normal, boxsize=0.7cm,centertableaux}
\begin{abstract}
  We resolve a question posed by Benedetti and Sagan by constructing
  a sign-reversing involution on Takeuchi's expansion that yields
  the antipode for the ring of symmetric functions in terms of the Schur basis.
\end{abstract}

\maketitle
\section{Introduction}
Many combinatorial objects admit two natural operations: joining and restriction.
Such operations often endow these objects with a natural Hopf algebra structure,
revealing rich combinatorial and algebraic properties.
The study of combinatorial Hopf algebras has thus become a central theme in
algebraic combinatorics.
A fundamental and distinguished component of a Hopf algebra is its antipode.
It plays a central role both in the theory and in applications of Hopf algebras.
Takeuchi~\cite{Takeuchi1971} showed that every connected graded bialgebra
\( (H, m, u, \Delta, \epsilon) \) is a Hopf algebra, with antipode \( S \) given by
\begin{equation} \label{eq:Takeuchi}
  S = \sum_{k \ge 0}(-1)^km^{k-1}\pi^{\otimes k}\Delta^{k-1}.
\end{equation}
Here \(m^{k-1}\) (resp., \(\Delta^{k-1}\)) denotes the \((k-1)\)-fold product (resp., coproduct), and we adopt the convention \(m^{-1}=u\), \(\Delta^{-1}=\epsilon\).
Moreover, \(\pi\) is the projection with \(\pi|_{H_0}=0\) and \(\pi|_{\oplus_{n\ge1}H_n}=\mathrm{id}\).
Although this expression is completely general, it typically involves extensive cancellation and is rarely practical for explicit computations.
To handle this issue, Benedetti and Sagan~\cite{Benedetti2017} employed a well-structured
combinatorial method, constructing a sign-reversing involution that provides
an elegant way to extract cancellation-free formulas from Takeuchi's expansion.
They successfully applied this technique to numerous combinatorial Hopf algebras.
Their proofs do not use algebraic machinery, but rely only on elementary
combinatorial features of each Hopf algebra, and thus help us understand these
Hopf algebras at a fundamental level.

Perhaps the most prominent instance of a Hopf algebra in combinatorics is the ring \( \Sym \) of
symmetric functions, which occupies a central position in algebraic combinatorics,
serving as a universal object that connects representation theory, algebraic geometry,
and other areas.
Moreover, among its many distinguished bases, the \emph{Schur basis}
\( \{s_\lambda\} \) stands out as the most fundamental and important.
The antipode formula for Schur functions is well-known:
\begin{thm}[\cite{Hall1959}] \label{thm:antipode_and_involution}
  The antipode in \( \Sym \) is given by
  \begin{equation}\label{eq:antipode}
    S(s_{\lambda/\mu}) = (-1)^{|\lambda|-|\mu|}s_{\lambda^t/\mu^t},
  \end{equation}
  where \( \lambda^t/\mu^t \) denotes the conjugate skew shape.
\end{thm}
However, existing proofs of the formula rely on algebraic identities or properties
of the \( \omega \) involution, rather than providing a direct combinatorial derivation
from Takeuchi's formula.
The absence of such an elementary proof has left a gap in the theory,
and Benedetti and Sagan posed the following question:
\begin{ques}[\cite{Benedetti2017}] \label{ques}
  Is there a sign-reversing involution that yields a cancellation-free formula
  for the antipode of \( \Sym \) when expressed in the Schur basis?
\end{ques}

In this paper, we answer the question by constructing a natural sign-reversing
involution \( \Phi \) in Takeuchi's expansion that yields the antipode
formula~\eqref{eq:antipode} in a purely combinatorial manner,
thereby completing the picture of sign-reversing involution proofs for antipodes
in the major combinatorial Hopf algebras.

\section{Proof of Theorem}
We follow the convention used in \cite{Grinberg2014}, and all the material
we use below can be found therein.

It is well known that \( \Sym \) has a graded connected Hopf algebra structure.
The coproduct \( \Delta \) on Schur functions is given by
\[
  \Delta(s_{\lambda / \mu}) = \sum_{\mu \subseteq \nu \subseteq \lambda} s_{\nu / \mu} \otimes s_{\lambda / \nu}.
\]
Thus, we have
\[
  \Delta^{k-1}(s_{\lambda / \mu}) = \sum_{\mu = \lambda_0 \subseteq \lambda_1 \subseteq \cdots \subseteq \lambda_k = \lambda}
  s_{\lambda_1 / \lambda_0} \otimes s_{\lambda_2 / \lambda_1} \otimes \cdots \otimes s_{\lambda_k / \lambda_{k-1}}
\]
for any \( k \ge 1 \).
Applying Takeuchi's formula~\eqref{eq:Takeuchi}, we see that the projection
\( \pi \) enforces strict inclusions:
\begin{equation}\label{eq:antipode=sum-takeuchi}
S(s_{\lambda/\mu})=  \sum_{k \ge 0}\sum_{\mu=\lambda_0 \subsetneq \lambda_1 \subsetneq \cdots \subsetneq \lambda_k = \lambda}(-1)^k
  s_{\lambda_1/\lambda_0} s_{\lambda_2 / \lambda_1} \cdots s_{\lambda_k / \lambda_{k-1}}.
\end{equation}
Since this expansion typically produces many cancelling terms, we construct an involution to derive~\eqref{eq:antipode} from Takeuchi's expansion~\eqref{eq:antipode=sum-takeuchi}.

Recall that \( s_{\lambda/\mu} \) is the generating function of semistandard Young tableaux of shape \( \lambda/\mu \) weighted by content:
\[
  s_{\lambda/\mu} = \sum_{T\in\SSYT(\lambda/\mu)} x^{T},
\]
where \( \SSYT(\lambda/\mu) \) denotes the set of semistandard
Young tableaux of shape \( \lambda/\mu \),
and \( x^T = \prod_{c\in T} x_{T(c)} \) where \( T(c) \) is the integer in the cell
\( c \) of \( T \).
Thus, \eqref{eq:antipode=sum-takeuchi} can be viewed as a signed sum
of products of generating functions of semistandard tableaux.
For partitions \(\mu\subseteq\lambda \), let
\[
  \calC_{\mu}^\lambda := \bigsqcup_{k\ge 0}~\{(\lambda_0,\dots,\lambda_k) \mid
  \mu=\lambda_0\subsetneq \lambda_1\subsetneq \cdots \subsetneq \lambda_k=\lambda \},
\]
and
\[
  X_\mu^\lambda :=
    \bigsqcup_{(\lambda_0,\dots,\lambda_k)\in\calC_{\mu}^\lambda}
    ~\prod_{i=1}^{k}\SSYT(\lambda_i/\lambda_{i-1}).
\]
For each element \( T=(T^{(1)},T^{(2)},\ldots,T^{(k)})\in X_\mu^\lambda \),
we define its \emph{length} to be \( \ell(T)=k \) and its \emph{weight}
to be \( \wt(T)=x^{T^{(1)}}x^{T^{(2)}}\cdots x^{T^{(k)}} \).
Then \eqref{eq:antipode=sum-takeuchi} becomes
\begin{equation}\label{eq:antipode=sum_X}
S(s_{\lambda/\mu})=  \sum_{T\in X_\mu^\lambda}(-1)^{\ell(T)}\wt(T).
\end{equation}
We construct a sign-reversing, weight-preserving involution \( \Phi \) on \( X_\mu^\lambda \) that yields~\eqref{eq:antipode}.
Given an element \( (T^{(1)},\dots,T^{(k)}) \in X_\mu^\lambda \), we obtain
a tableau \( T \) of shape \( \lambda/\mu \) by concatenating the tableaux \( T^{(i)} \).
For simplicity, we write \( T = (T^{(1)},T^{(2)},\ldots,T^{(k)}) \).
Note that in general \( T \) is not semistandard.
We now define a total order on the cells of \( T \): for cells \( c = (i, j) \) and
\( c' = (i', j') \) in \( T \),  we write \( c < c' \) if either
\begin{itemize}
  \item[(i)] \( T(c) < T(c') \);
  \item[(ii)] \( T(c) = T(c') \) and \( j < j' \); or
  \item[(iii)] \( T(c) = T(c') \), \( j = j' \) and \( i > i' \).
\end{itemize}

A cell \( c \) of \( T \) with \( c\in T^{(i)} \) is called \emph{splittable}
if \( |T^{(i)}| > 1 \) and \( c \) is the largest cell in \( T^{(i)} \),
and is called \emph{mergeable} if \( i > 1 \), \( |T^{(i)}| = 1 \),
\( T^{(i-1)}\sqcup T^{(i)} \) is semistandard, and \( c \) is larger than
any cell in \( T^{(i-1)} \).
We denote by \( \cell(T) \) the largest cell among splittable or mergeable cells of \( T \).
If no such cell exists, we set \( \cell(T)=\emptyset \).
We say that \( T \) is \emph{splittable} if \( \cell(T) \) is splittable,
and \emph{mergeable} if \( \cell(T) \) is mergeable.
\begin{defn} \label{def:involution}
  Define the map \( \Phi : X_\mu^\lambda \to X_\mu^\lambda \) as follows.
  For \( T=(T^{(1)},\ldots,T^{(k)})\in X_\mu^\lambda \),
  \[
    \Phi(T) :=
    \begin{cases}
      (T^{(1)},\ldots,T^{(i-1)},T^{(i)}\setminus \{\cell(T)\},T',T^{(i+1)},\ldots,T^{(k)}), & \text{if \( T \) is splittable} \\
      (T^{(1)},\ldots,T^{(i-1)}\sqcup T^{(i)},T^{(i+1)},\ldots,T^{(k)}), & \text{if \( T \) is mergeable} \\
      T, & \text{if}\ \cell(T)=\emptyset
    \end{cases}
  \]
  where \( T^{(i)} \) contains \( \cell(T) \), and \( T' \) is the tableau
  consisting of the single cell \( \cell(T) \).
\end{defn}

\begin{lem}\label{lem:well-defined}
The map \( \Phi \) is well-defined.
\end{lem}
\begin{proof}
  There is nothing to verify when \( T \) is mergeable,
  so it suffices to show that when \( T \) is splittable,
  \begin{equation} \label{eq:splitting}
    (T^{(1)},\ldots,T^{(i-1)},T^{(i)}\setminus \{\cell(T)\},T',T^{(i+1)},\ldots,T^{(k)})
    \in X^\lambda_\mu.
  \end{equation}
  By the definition of the total order and the fact that \( T^{(i)} \) is semistandard
  and \( \cell(T) \) is largest in \( T^{(i)} \),
  there is no cell below or to the right of \( \cell(T) \) in \( T^{(i)} \),
  which implies that \( \cell(T) \) is a corner cell of \( T^{(i)} \).
  Then \( T^{(i)}\setminus\{\cell(T)\} \) is a valid skew shape, and thus
  we deduce \eqref{eq:splitting}.
\end{proof}

\begin{lem} \label{lem:involution}
  The map \( \Phi \) is a sign-reversing, weight-preserving involution.
\end{lem}
\begin{proof}
  Let \( T=(T^{(1)},\ldots,T^{(k)})\in X_\mu^\lambda \) with \( \cell(T) \neq \emptyset \),
  and let \( c = \cell(T) \in T^{(i)} \).
  Since splitting and merging are inverse to each other,
  it is enough to show that \( \cell(\Phi(T)) = c \).

  We first suppose that \( T \) is splittable.
  After splitting, we obtain
  \[
    \Phi(T) = (T^{(1)}, \ldots, T^{(i-1)}, T^{(i)} \setminus \{c\}, T', T^{(i+1)}, \ldots, T^{(k)}),
  \]
  where \( T' = \{c\} \). Since \( |T'| = 1 \), \( T^{(i)} = (T^{(i)}\setminus \{c\}) \sqcup T' \)
  is semistandard, and \( c \) is largest in \( T^{(i)} \),
  the cell \( c \) is mergeable in \( \Phi(T) \).
  So we only need to check that there is no splittable or mergeable cell \( c' > c \)
  in \( \Phi(T) \).
  Suppose that such \( c' \) exists. Since \( c \) is largest in \( T^{(i)} \),
  the cell \( c' \) must be contained in \( T^{(j)} \) for some \( j\neq i \).
  Moreover, \( c' \) is neither splittable nor mergeable in \( T \) because \( c=\cell(T) \)
  is largest among splittable or mergeable cells in \( T \).
  Recall that whether \( c' \) is splittable (resp., mergeable) or not depends only on \( T^{(j)} \)
  (resp., \( T^{(j-1)} \) and \( T^{(j)} \)).
  From these facts, it follows that \( j = i+1 \) and \( c' \) is mergeable in
  \( \Phi(T) \) so that \( T^{(i+1)} = \{c'\} \) and \( T'\sqcup T^{(i+1)} \) is semistandard.
  However, in this case, \( c' \) is also mergeable in \( T \), which is a contradiction.
  We therefore deduce \( c = \cell(\Phi(T)) \), as desired.

  We now assume that \( T \) is mergeable. Then we have
  \[
    \Phi(T) = (T^{(1)}, \ldots, T^{(i-1)}\sqcup T^{(i)}, T^{(i+1)}, \ldots, T^{(k)}),
  \]
  where \( T^{(i)} = \{c\} \). Then by definition, \( c \) is splittable in \( \Phi(T) \).
  In addition, if \( c' \neq c \) is splittable or mergeable in \( \Phi(T) \),
  then so is \( c' \) in \( T \). Therefore, we obtain \( \cell(\Phi(T)) = c \),
  which completes the proof of \( \Phi(\Phi(T)) = T \).
  
  Furthermore, since the splitting and merging processes change the length of \( T \) by 1, and
  the weight \( \wt(T) \) depends only on integers in cells of \( T \), the involution \( \Phi \)
  is sign-reversing and weight-preserving.
\end{proof}

We illustrate the involution with three examples in Figure~\ref{fig:examples},
where the small number in the top-left corner of each cell indicates the index
\( i \) where the tableau \( T^{(i)} \) contains that cell.
For \( T_1 \), we have \( \cell(T_1) = (2, 2) \), which contains the entry~5. This cell is the unique cell of \( T^{(4)} \), and since \( T^{(3)} \sqcup T^{(4)} \) is semistandard, \( T_1 \) is mergeable.
For \( T_2 \), we have \( \cell(T_2) = (2, 2) \). Since there are other cells contained in
\( T^{(3)} \), \( T_2 \) is splittable. Moreover, the involution \( \Phi \) maps \( T_1 \)
to \( T_2 \), and vice-versa.
On the other hand, \( T_3 \) has no splittable or mergeable cell,
so \( \cell(T_3) = \emptyset \) and \( T_3 \) is fixed by \( \Phi \).

\begin{figure}
  \centering
  \begin{tikzpicture}[scale=0.8, line width=0.8pt]
    \begin{scope}[shift={(0,0)}]
      \brick{0}{0}{1}
      \brick{1}{0}{1}
      \brick{1}{1}{1}
      \brick{2}{1}{1}
      \brick{3}{1}{1}
      \brick{2}{0}{1}
      \node at (2, -0.75) {\(T_1\)};
      \node at (0.5, 0.5) {5};
      \node at (1.5, 0.5) {5};
      \node at (2.5, 0.5) {1};
      \node at (1.5, 1.5) {5};
      \node at (2.5, 1.5) {2};
      \node at (3.5, 1.5) {4};
      \tikzset{font={\fontsize{7pt}{14pt}\selectfont}}
      \node at (1.2, 1.8) {1};
      \node at (0.2, 0.8) {2};
      \node at (2.2, 1.8) {3};
      \node at (3.2, 1.8) {3};
      \node at (1.2, 0.8) {4};
      \node at (2.2, 0.8) {5};
    \end{scope}
    \begin{scope}[shift={(6,0)}]
      \brick{0}{0}{1}
      \brick{1}{0}{1}
      \brick{1}{1}{1}
      \brick{2}{1}{1}
      \brick{3}{1}{1}
      \brick{2}{0}{1}
      \node at (2, -0.75) {\(T_2\)};
      \node at (0.5, 0.5) {5};
      \node at (1.5, 0.5) {5};
      \node at (2.5, 0.5) {1};
      \node at (1.5, 1.5) {5};
      \node at (2.5, 1.5) {2};
      \node at (3.5, 1.5) {4};
      \tikzset{font={\fontsize{7pt}{14pt}\selectfont}}
      \node at (1.2, 1.8) {1};
      \node at (0.2, 0.8) {2};
      \node at (2.2, 1.8) {3};
      \node at (3.2, 1.8) {3};
      \node at (1.2, 0.8) {3};
      \node at (2.2, 0.8) {4};
    \end{scope}
    \begin{scope}[shift={(12,0)}]
      \brick{0}{0}{1}
      \brick{1}{0}{1}
      \brick{1}{1}{1}
      \brick{2}{1}{1}
      \brick{3}{1}{1}
      \brick{2}{0}{1}
      \node at (2, -0.75) {\(T_3\)};
      \node at (0.5, 0.5) {4};
      \node at (1.5, 0.5) {3};
      \node at (2.5, 0.5) {2};
      \node at (1.5, 1.5) {5};
      \node at (2.5, 1.5) {2};
      \node at (3.5, 1.5) {1};
      \tikzset{font={\fontsize{7pt}{14pt}\selectfont}}
      \node at (1.2, 1.8) {1};
      \node at (0.2, 0.8) {2};
      \node at (2.2, 1.8) {4};
      \node at (3.2, 1.8) {6};
      \node at (1.2, 0.8) {3};
      \node at (2.2, 0.8) {5};
    \end{scope}
  \end{tikzpicture}
  \caption{Examples of the involution \( \Phi \)}
  \label{fig:examples}
\end{figure}
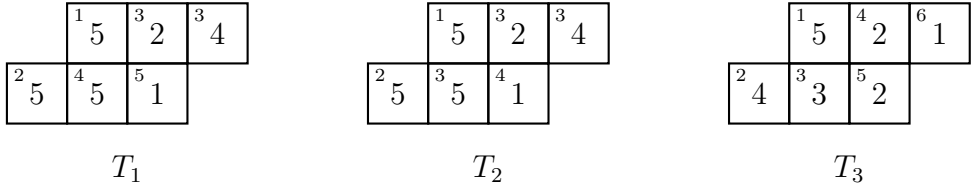

We now characterize fixed points of \( \Phi \).
\begin{lem}\label{lem:fixed}
  For \( T=(T^{(1)},T^{(2)},\ldots,T^{(\ell(T))})\in X_\mu^\lambda \),
  \( T \) is a fixed point of \( \Phi \) if and only if the following conditions hold:
\begin{itemize}
\item[(F1)]  \( \ell(T)=|\lambda|-|\mu| \), i.e., \( |T^{(i)}|=1 \) for all \( i \); and
\item[(F2)] enumerating the cells \( c_1<c_2< \cdots <c_{\ell(T)}\) in \( T \),
  we have \( c_{\ell(T)-i+1}\in T^{(i)} \) for all \( i \).
\end{itemize}
\end{lem}

\begin{proof}
  By definition, \( T \) is fixed by \( \Phi \) if and only if \( \cell(T) = \emptyset \),
  which holds if and only if \( T \) is neither splittable nor mergeable.
  Since condition~(F1) is equivalent to \( T \) not being splittable,
  it suffices to show that condition~(F2) is equivalent to \( T \) not being mergeable.
  For \( i=2,\dots,\ell(T) \), let \( c \) and \( c' \) be the cells such that
  \( T^{(i-1)}=\{c\} \) and \( T^{(i)}=\{c'\} \).
  If \( c < c' \), then \( T^{(i-1)}\sqcup T^{(i)} \) is semistandard by
  the definition of the total order, and thus \( c' \) is mergeable.
  If \( c > c' \), then \( c' \) is not mergeable by the definition of a mergeable cell.
  We conclude that \( c' \) is not mergeable if and only if \( c > c' \).
  Consequently, \( T \) is not mergeable if and only if condition~(F2) holds,
  which completes the proof.
\end{proof}

Let \( \Fix(\lambda/\mu) \) be the set of fixed points of \( \Phi \).
From \eqref{eq:antipode=sum_X}, combining Lemmas~\ref{lem:involution} and \ref{lem:fixed},
we derive
\begin{equation}\label{eq:antipode=sum_fixed}
  S(s_{\lambda/\mu})=(-1)^{|\lambda|-|\mu|} \sum_{T\in\Fix(\lambda/\mu)}\wt(T).
\end{equation}
We now provide another description of \( \Fix(\lambda/\mu) \).
A \emph{row-strict plane partition} is a plane partition with weakly decreasing
columns and strictly decreasing rows; see Figure~\ref{fig:RSPP+SSYT}.
\begin{figure}
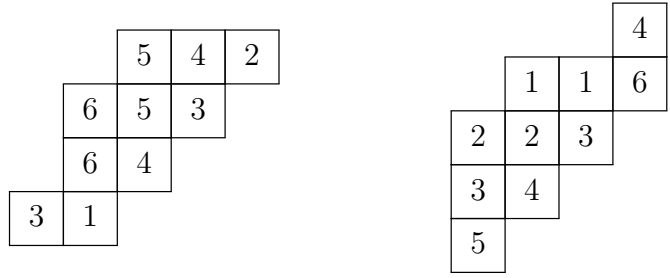

  \begin{center}
    \begin{ytableau}
      \none & \none  & 5 & 4 & 2 \\  \none & 6 & 5 & 3 \\ \none & 6 & 4 \\ 3 & 1
    \end{ytableau}
    \hspace{20mm}
    \begin{ytableau}
      \none & \none & \none & 4 \\ \none & 1 & 1 & 6 \\ 2 & 2 & 3 \\ 3 & 4 \\ 5
    \end{ytableau}
  \end{center}
  \caption{A row-strict plane partition and a semistandard Young tableau}
  \label{fig:RSPP+SSYT}
\end{figure}

There is a natural bijection between \( \Fix(\lambda/\mu) \)
and the set \( \RSPP(\lambda/\mu) \) of row-strict plane partitions of shape
\( \lambda/\mu \). For any \( T=(T^{(1)},T^{(2)},\ldots,T^{(\ell(T))})\in\Fix(\lambda/\mu) \),
condition~(F2) and the definition of the total order ensure that \( T=\bigsqcup_i T^{(i)} \)
is a row-strict plane partition.
Conversely, given a row-strict plane partition of shape \( \lambda/\mu \),
enumerate all cells with respect to the total order:
\[
  c_1<c_2< \cdots < c_{|\lambda|-|\mu|}.
\]
Setting \( T^{(i)} \) to be the singleton tableau containing only \( c_{|\lambda|-|\mu|-i+1} \),
we obtain an element \( T=(T^{(1)},T^{(2)},\ldots,T^{(|\lambda|-|\mu|)}) \in \Fix(\lambda/\mu) \),
by Lemma~\ref{lem:fixed}.
Thus, \eqref{eq:antipode=sum_fixed} becomes
\begin{equation}\label{eq:antipode=RSPP}
  S(s_{\lambda/\mu})=(-1)^{|\lambda|-|\mu|} \sum_{T\in\RSPP(\lambda/\mu)}x^T.
\end{equation}

We note that, upon reversing the order on the integers,
the row and column conditions of row-strict plane partitions become precisely
the column and row conditions of semistandard Young tableaux, respectively;
see Figure~\ref{fig:RSPP+SSYT}.
Using this observation and the symmetry of Schur functions,
the generating function of row-strict plane partitions can be expressed as
\begin{equation}\label{eq:schur=RSPP}
  \sum_{T\in\RSPP(\lambda/\mu)}x^T
    = s_{\lambda^t/\mu^t}(\dots,x_2,x_1)
    = s_{\lambda^t/\mu^t}(x_1,x_2,\dots).
\end{equation}
Combining \eqref{eq:antipode=RSPP} and \eqref{eq:schur=RSPP}, we obtain
\[
  S(s_{\lambda/\mu})= (-1)^{|\lambda|-|\mu|}s_{\lambda^t/\mu^t},
\]
which completes the proof and answers the question of Benedetti and Sagan.

\section*{Acknowledgements}
Y. Cho and H. Lee were supported by the National Research Foundation of Korea (NRF) grant funded by the Korea government RS-2025-00557835.
B.-H. Hwang was supported by a KIAS Individual Grant (AP098201) via the Center for Artificial Intelligence and Natural Sciences at Korea Institute for Advanced Study.


\begin{thebibliography}{1}

\bibitem{Benedetti2017}
C.~Benedetti and B.~E. Sagan.
\newblock Antipodes and involutions.
\newblock {\em Journal of Combinatorial Theory, Series A}, 148:275--315, 2017.

\bibitem{Grinberg2014}
D.~Grinberg and V.~Reiner.
\newblock {Hopf Algebras in Combinatorics}.
\newblock {\it Preprint}, \href{https://arxiv.org/abs/1409.8356}{arXiv:1409.8356}, 2014.

\bibitem{Hall1959}
P.~Hall.
\newblock The algebra of partitions.
\newblock {\em Proceedings 4th Canadian Mathematical Congress}, pages 147--159, 1959.

\bibitem{Takeuchi1971}
M.~Takeuchi.
\newblock Free hopf algebras generated by coalgebras.
\newblock {\em Journal of the Mathematical Society of Japan}, 23(4):561--582, 1971.

\end{thebibliography}
\end{document}